\documentclass[]{article}
\usepackage{amsmath,amsthm,amssymb,amsfonts,amscd} 
\usepackage{color}
\usepackage{comment}
\usepackage{paralist}
\usepackage{mathrsfs}
\usepackage{hyperref} 
\usepackage{setspace}
\usepackage{xcolor}
\usepackage{titlesec}
\usepackage{authblk}
\usepackage{graphicx}
\usepackage{subcaption}
\usepackage[normalem]{ulem}
\titleformat{\subsection}{\normalfont\scshape\filcenter}{\thesubsection}{1em}{}
\titleformat{\subsubsection}[runin]{\normalfont\bfseries}{\thesubsubsection. }{.2em}{}[.]

\newcommand\thefontsize[1]{{#1 The current font size is: \f@size pt\par}}

\newtheorem{lem}{Lemma}[section]
\newtheorem{define}{Definition}[section]
\newtheorem{remark}{Remark}[section]
\newtheorem{cor}[lem]{Corollary}

\newcommand{\R}{{\mathbb R}}
\newcommand{\N}{{\mathbb{N}}}

\newcommand{\Q}{{\mathbb Q}}

\newcommand{\eps}{\varepsilon}

\newcommand{\al}{\alpha}

\newcommand{\lra}{\leftrightarrow}

\newcommand{\p}{\partial}

\newcommand{\pt}{\tilde{p}}
\newcommand{\ue}{\underline{E}}
\newcommand{\inner}{\text{inner}}
\DeclareMathOperator{\GP}{GP}
\newcommand{\easy}{\textsc{easy}}
\newcommand{\hard}{\textsc{hard}}
\newcommand{\rel}{\textsc{relevant}}

\newcommand{\uz}{\underline{z}}
\newcommand{\ux}{\underline{x}}
\newcommand{\Qbig}{Q_{\textrm{big}}}
\newcommand{\Qsmall}{Q_{\textrm{small}}}
\newcommand{\Qinner}{Q_{\textrm{inner}}}
\newcommand{\norm}[1]{\Vert #1 \Vert}

\newcommand{\hL}{\hat{L}}
\newcommand{\tF}{\tilde{F}}

\DeclareMathOperator{\diam}{diam}
\DeclareMathOperator{\dist}{dist}
\DeclareMathOperator{\supp}{supp}

\numberwithin{equation}{section}

\newcommand{\Ecal}{\mathcal{E}}
\newcommand{\Mcal}{\mathcal{M}}

\DeclareMathOperator{\Av}{Av}
\DeclareMathOperator{\Gp}{Gp}
\DeclareMathOperator{\Gang}{Gang}
\DeclareMathOperator{\CZ}{CZ}
\DeclareMathOperator{\DQ}{DQ}

\newtheorem{theorem}{Theorem}

\newcommand{\vertiii}[1]{{\left\vert\kern-0.25ex\left\vert\kern-0.25ex\left\vert #1 
    \right\vert\kern-0.25ex\right\vert\kern-0.25ex\right\vert}}
\newcommand{\triplenorm}[1]{\vertiii{#1}}

\title{Sobolev extension in a simple case}

\author{Marjorie Drake\thanks{mkdrake@mit.edu, supported by NSF Award No. 2103209.}}
\author{Charles Fefferman\thanks{cf@math.princeton.edu, supported by the AFOSR through the grant FA9550-23-1-0273.}}
\author{Kevin Ren \thanks{kr5621@math.princeton.edu, supported by a NSF GRFP fellowship.}}
\author{Anna Skorobogatova \thanks{as110@princeton.edu, supported by the NSF through the grant FRG-1854147.}}
\affil{\textsuperscript{$*$} Department of Mathematics, Massachusetts Institute of Technology, 77 Massachusetts Ave., Cambridge,
MA 02139, USA}
\affil{\textsuperscript{$\dagger,\ddagger,\S$} Department of Mathematics, Fine Hall, Princeton University, Washington Road, Princeton, NJ 08540, USA}

\date{} 
\allowdisplaybreaks

\begin{document}
\maketitle

\centering Dedicated to Bob Fefferman
\flushleft
\vspace{20pt}

\begin{abstract}
    In this paper, we establish the existence of a bounded, linear extension operator $T: L^{2,p}(E) \to L^{2,p}(\mathbb{R}^2)$ when $1<p<2$ and $E$ is a finite subset of $\mathbb{R}^2$ contained in a line.
\end{abstract}


keywords: Sobolev extension, linear extension, Whitney extension, Calder\'{o}n-Zygmund decomposition

\tableofcontents

\section{Introduction}
\label{sec:intro}

Continuing from \cite{arie3,arie4,arie5,arie6,arie7,arie8}, we study the \emph{Sobolev extension problem}: given a finite set $E\subset \R^n$ and a function $f:E\to \R$, our task is to find a function $F$ in the homogeneous Sobolev space $L^{m,p}(\R^n)$ satisfying $F=f$ on $E$ such that $\Vert F\Vert_{L^{m,p}(\R^n)}$ has the least possible order of magnitude. Moreover, we want to compute the order of magnitude of $\Vert F\Vert_{L^{m,p}(\R^n)}$. Here, $L^{m,p}(\R^n)$ denotes the homogeneous Sobolev spaces consisting of functions $F$ on $\R^n$ whose distributional derivatives of order $m$ belong to $L^p$, with seminorm 
\[
    \Vert F\Vert_{L^{m,p}(\R^n)} := \left(\sum_{|\alpha|=m} \int_{\R^n} |\partial^\alpha F|^p\right)^{1/p}.
\]
Notice that when $p>\frac{n}{m}$, any function $F\in L^{m,p}(\R^n)$ is continuous, so the restriction $F|_E$ to any finite subset $E\subset \R^n$ is well-defined. Thus, our requirement that $F=f$ on $E$ makes sense for any choice of $f$ and our problem is well-posed provided that $p>\frac{n}{m}$.

In the narrower parameter range $p>n$, the problem is well-understood; see Israel \cite{arie5}, Fefferman-Israel-Luli \cite{arie3,arie4} and Drake \cite{Drake}. We recall the relevant results as follows. Let $L^{m,p}(E)$ denote the space of functions $f:E\to\R$ equipped with the seminorm
\[
    \Vert f \Vert_{L^{m,p}(E)} := \inf\{ \Vert F \Vert_{L^{m,p}(\R^n)} : F\in L^{m,p}(\R^n), \ \text{$F=f$ on $E$}\}.
\]

{ For $p>n$ and $E \subset \R^n$ finite, there exist a linear map $T: L^{m,p}(E) \to L^{m,p}(\R^n)$, and linear functionals $\ell_1, ..., \ell_{\nu_{\max}}: L^{m,p}(E) \to \R$, with the following properties:
\begin{itemize}
    \item[(1)] $Tf=f \text{ on }E \text{ and } \Vert Tf\Vert_{L^{m,p}(\R^n)}\leq C_{m,n,p} \Vert f\Vert_{L^{m,p}(E)} \text{ for all }f\in L^{m,p}(E).$
    \item[(2)] Let $\vertiii{f}_p:= \left(\sum_{\nu=1}^{\nu_{\max}} |\ell_\nu(f)|^p\right)^{1/p}$. Then
\begin{align*}
    &c_{m,n,p} \vertiii{f}_p \leq \Vert f \Vert_{L^{m,p}(E)} \leq C_{m,n,p} \vertiii{f}_p \text{ for all }f\in L^{m,p}(E).
\end{align*}
Here, the constants $C_{m,n,p}$ and $c_{m,n,p}$ depend only on $m,n,p$, the parameters specifying the Sobolev space.
\item[(3)] 
Moreover, the number $\nu_{\max}$ of linear functionals $\ell_\nu$ grows only linearly with $N$, the number of points of $E$, and the $\ell_{\nu}$ have a sparse structure that lets us evaluate $\vertiii{f}_p$ in $O(N)$ computer operations. The functionals $\ell_\nu$ and the operator $T$ can be computed efficiently. See \cite{arie3} for details.
\end{itemize}
}

When $\frac{n}{m} < p \leq n$, the proofs of the above results break down, and we know almost nothing. In this paper, we analyze the simplest non-trivial case. We work in the homogeneous Sobolev space $L^{2,p}(\R^2)$ for $1<p<2$. We suppose our finite set $E$ is contained in a line. Let $(x,y)$ denote the canonical orthogonal coordinates in $\R^2$; by rotating our coordinates, we may assume that $E$ lies on the $x$-axis. Our main results are as follows.

Fix $p$ ($1<p<2$). Let $E\subset \R\times \{0\}\subset \R^2$ be finite, with $N=\# E \geq 2$. We write $c, C, C',$ etc. to denote absolute constants, and we write $c_p, C_p, C_p'$, etc. to denote constants depending only on $p$.

\begin{theorem} \label{t:main1}
    There exists a linear map $T:L^{2,p}(E)\to L^{2,p}(\R^2)$ such that $Tf=f$ on $E$ and $\Vert Tf \Vert_{L^{2,p}(\R^2)}\leq C_p \Vert f \Vert_{L^{2,p}(E)}$ for every $f\in L^{2,p}(E)$. Moreover, we can take $T$ independent of $p$, such that for each point $z \in \R^2$, the value of $Tf$ at $z$ is determined by the value of $f$ at at most $C$ points of $E$.
\end{theorem}

Further, under the same assumptions as in Theorem \ref{t:main1}, we can efficiently compute the order of magnitude of $\Vert f \Vert_{L^{2,p}(E)}$:
\begin{theorem}\label{t:main2}
    There exist positive constants $\lambda_1,\dots,\lambda_{\nu_{\max}}\in \R$ and linear functionals $\ell_1,\dots,\ell_{\nu_{\max}}:L^{2,p}(E)\to \R$ such that the quantity $\vertiii{f}_p:= \left(\sum_{\nu=1}^{\nu_{\max}} \lambda_\nu |\ell_\nu(f)|^p\right)^{1/p}$ satisfies
    \[
        c_p \vertiii{f}_p \leq \Vert f \Vert_{L^{2,p}(E)} \leq C_p \vertiii{f}_p
    \]
    for all $f\in L^{2,p}(E)$. 
    
    Moreover, we can take  $\lambda_\nu, \ell_\nu$ independent of $p$ such that $\nu_{\max} \leq CN$ for $N = \# E$, and for each $\nu$, the quantity $\ell_\nu(f)$ is determined by the value of $f$ at at most $C$ points of $E$.
\end{theorem}

The remainder of the paper is dedicated to the proof of Theorem \ref{t:main1} and Theorem \ref{t:main2}.

\begin{remark}
    The conclusions of Theorem \ref{t:main1} and Theorem \ref{t:main2} remain open for an arbitrary choice of finite subset $E\subset \R^2$. See Carruth-Israel \cite{CI23} for a class of examples in which $E$ is contained in a union of two closely-spaced parallel lines and substantial new issues arise. See also Fefferman-Klartag \cite{FK23}.
\end{remark}

Before proceeding, let us first outline an important issue arising in the proof of Theorem \ref{t:main1} and Theorem \ref{t:main2} and the corresponding key new ingredient introduced here to deal with this. Note that this discussion is merely an overview; see the relevant succeeding sections for a more precise description. Fix $1<p<2$ and suppose we are given $f : E \to \R$. We will make a Calder\'{o}n-Zygmund (henceforth abbreviated by CZ) decomposition in the plane. Figure \ref{fig:CZ} depicts two possible scenarios for a CZ square $Q$, whose distance to the $x$-axis is comparable to its side length $\delta_Q$, together with nearby points in $E$. In Figure \ref{fig:cz-easy}, $|z_1 - z_2|$ is comparable to $\delta_Q$, while in Figure \ref{fig:cz-hard}, $|z_1-z_2|$ is much smaller than $\delta_Q$, while $|z^*-z_1|$ is much larger than $\delta_Q$.

\begin{figure}[h]
\caption{Two possible scenarios for a CZ square $Q$.}
\label{fig:CZ}
\begin{subfigure}{.35\textwidth}
\centering
\includegraphics[width=1.06\textwidth]{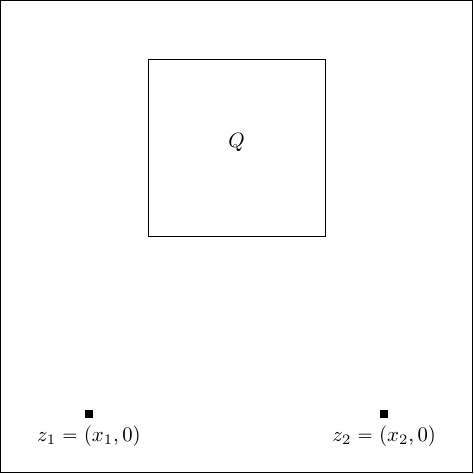}
\caption{$z_1,z_2\in E$ near $Q$ with $|z_1 - z_2|\simeq \delta_Q$.}
\label{fig:cz-easy}
\end{subfigure}
\begin{subfigure}{.7\textwidth}
\centering\includegraphics[width=0.8\textwidth]{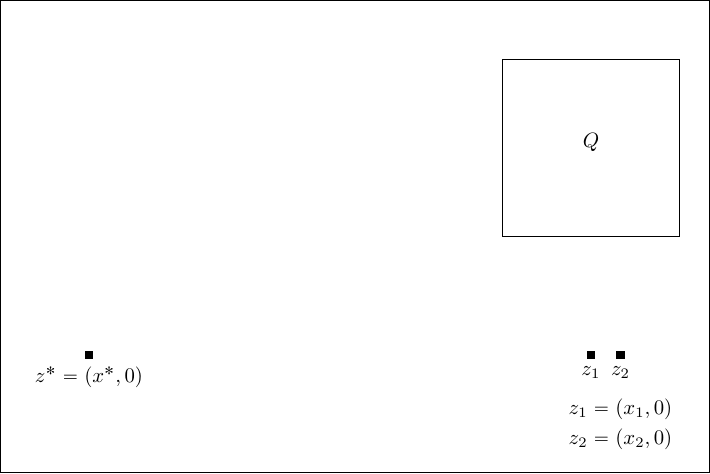}
\caption{$z_1,z_2,z^*\in E$ near $Q$ with $|z_1 - z_2|\ll \delta_Q$ and $|z^*-z_1|\gg\delta_Q$.}
\label{fig:cz-hard}
\end{subfigure}
\end{figure}
More precisely, the points $z_1,z_2, z^*$ belong to $E$ and all other points of $E$ lie much farther away from $Q$ than those depicted. Let $F\in L^{2,p}(\R^2)$ satisfy $F=f$ on $E$ with $\Vert F\Vert_{L^{2,p}(\R^2)}$ being of least possible order of magnitude. What can we infer about the behavior of $F$ on $Q$? In particular, what can we say about $\partial_x F$ on $Q$?

If $z_1,z_2$ and $Q$ are as in Figure \ref{fig:cz-easy}, one expects that $\partial_x F |_Q$ is closely approximated by the difference quotient
\[
    \DQ(z_1,z_2) := \frac{f(z_2)-f(z_1)}{x_2-x_1}.
\]
The comparability of $|z_1-z_2|$ to $\delta_Q$ combined with a Sobolev inequality allows one to control \\ $\Vert\partial_x F - \DQ(z_1,z_2)\Vert_{L^p(Q)}$ by the product $\delta_Q \cdot \Vert \nabla^2 F\Vert_{L^p(CQ)}$, where $CQ$ is the dilation of $Q$ about its center by an appropriate constant factor $C>1$ such that $CQ$ contains $z_1, z_2$.

On the other hand, suppose that $z_1,z_2,z^*$ and $Q$ are as in Figure \ref{fig:cz-hard}. This time, since $|z_1-z_2|\ll \delta_Q$, $\DQ(z_1,z_2)$ is not an effective approximation for $\partial_x F|_Q$. Instead, the difference quotient $\DQ(z_1,z^*)$ is a better choice of approximant for $\partial_x F|_Q$. We may again control $\Vert\partial_x F - \DQ(z_1,{z^*})\Vert_{L^p(Q)}$ by the product $\delta_Q \cdot \Vert \nabla^2 F\Vert_{L^p(Q^*)}$ for a square $Q^*$ containing $z_1,z^*$, but this time $Q^*$ is much larger than $Q$, since $|z^*-z_1| \gg \delta_Q$. Therefore, in order to avoid having unbounded overlaps when summing over all squares $Q$, we cannot simply consider a single such square $Q$ associated to $Q^*$.

Instead, for a fixed $z^* \in E$, we group together all CZ squares $Q$ for which a scenario like that in Figure \ref{fig:cz-hard} holds for $z^*$. We denote this subcollection of the CZ squares by $\GP(z^*)$. Thanks to a lemma in Section \ref{sec:gps}, the structure of the squares in $\GP(z^*)$ looks like that depicted in Figure \ref{fig:Gp}. 
% \note{In the Figure, maybe we should have $\#(3Q_1 \cap E) \leq 1$? Would it be difficult to move the cubes a bit away from the set $E$?}{\color{purple}A: hopefully this is better? M: Looks nice; thanks for doing this}

\begin{figure}[h]
\centering\includegraphics[width=0.5\textwidth]{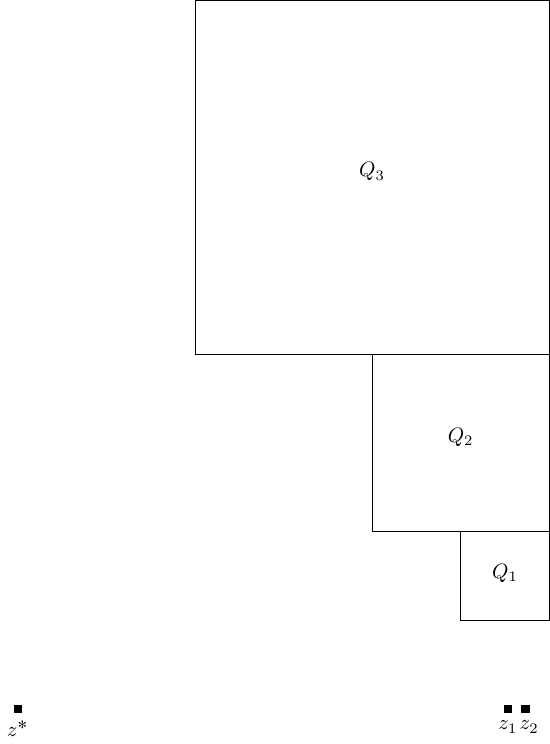}
\caption{CZ squares $Q_1,Q_2,Q_3\in \GP(z^*)$ for a point $z^*\in E$.}
\label{fig:Gp}
\end{figure}

This allows us to estimate
\[
    \sum_{Q\in\GP(z^*)} \delta_Q^{-p} \Vert \partial_x F - \DQ(z_1,z^*) \Vert_{L^p(Q)}^p
\]
by $\Vert \nabla^2 F\Vert_{L^p(Q^*)}^p$ for a fixed $z^*\in E$. This is a fundamental difference between the case $1<p<2$ in consideration here, and the aforementioned case $p>2$. Namely, when $p>2$ we can make use of the difference quotient $\DQ(z_1,z_2)$ even in the setting of Figure \ref{fig:cz-hard}, without taking into account $z^*$. This allows us to study each CZ square $Q$ in isolation, then sum over all $Q$ together.

So our main new ingredient is the introduction of $\GP(z^*)$. By making a translation and dilation of $E$, we may assume without loss of generality that $E=\underline{E}\times\{0\}\subset \R^2$ for a finite subset $\underline{E}$ with $\max \underline{E} = 2^{-11}$ and $\min\underline{E}=-2^{-11}$. We then work with the set $E^+=E\cup\{(-1,0)\}$, so that we may guarantee the existence of a point $z^*\in E$ as in Figure \ref{fig:cz-hard} to the left of any CZ square $Q$. At the end of the proof in Section \ref{sec:rempt}, we remove this additional point $(-1,0)$.

\subsection*{Acknowledgments}
We are grateful to Alper Gunes, Jonathan Marty, Arie Israel, Jacob Carruth and Ignacio Uriarte-Tuero for valuable conversations. We are grateful to the NSF and the AFOSR for their generous support.

\section{Main Result}

We will prove the following theorem, which is equivalent to Theorems \ref{t:main1} and \ref{t:main2}.

\begin{theorem}\label{t:main-reduction}
    Let $E\subset \R\times \{0\}\subset \R^2$ be finite with $N=\# E \geq 2$. Let $X(E)$ denote the vector space of real-valued functions on $E$. Then there exists a linear map ${T^\sharp}: X(E) \to \bigcap_{p\in(1,2)} L^{2,p}(\R^2)$, and there exist linear functionals $\ell_\nu : X(E) \to \R$ and positive coefficients $\lambda_\nu$ for $\nu=1,\dots,\nu_{\max}$, with the following properties:
    \begin{itemize}
        \item[(a)] For every $f\in X(E)$, ${T^{\sharp}}f=f$ on $E$;
        \item[(b)] For every $f\in X(E)$ and $p\in (1,2)$,
        \[
            \Vert {T^{\sharp}}f\Vert_{L^{2,p}(\R^2)}^p \leq C_p \sum_{\nu=1}^{\nu_{\max}} \lambda_\nu |\ell_\nu(f)|^p,
        \]
        {where the constant $C_p$ depends only on $p$;}
        \item[(c)] For $f\in X(E)$, $p\in (1,2)$ and $F\in L^{2,p}(\R^2)$ with $F=f$ on $E$, we have
        \[
            \Vert F \Vert_{L^{2,p}(\R^2)}^p \geq c_p \sum_{\nu=1}^{\nu_{\max}} \lambda_\nu |\ell_\nu(f)|^p,
        \]
        {where the constant $c_p$ depends only on $p$;}
        \item[(d)] For every point $z\in \R^2$, there exists a subset $S(z)\subset E$ with $\# S(z) \leq C$, such that ${T^\sharp} f(z)$ depends only on $f|_{S(z)}$;
        \item[(e)] For each $\nu \in \{1,\dots,\nu_{\max}\}$, there exists a subset $S_\nu\subset E$ with $\# S_\nu\leq C$, such that $\ell_\nu(f)$ depends only on $f|_{S_\nu}$;
        \item[(f)] $\nu_{\max}\leq CN$.
    \end{itemize}
    {The constants $C$ in (d), (e), and (f), are universal, so independent of the set $E$ and the function $f$.}
\end{theorem}

\section{Constants} \label{sec:constants}
% \note{I think this section should go before we introduce any constants, so somewhere in the intro, or at least just before the above theorem; M: I think we should leave this here, as it's about the contents of our proof but add a remark on the constants to the above theorem. I've added such a remark.} {\color{purple}A: this works, thanks!}
Let $1<p<2$. Fix $\tilde p= \frac{1+p}{2}\in(1,p)$. We write $c,C,C'$ to denote absolute constants and we write $c_p,C_p,C'_p$ to denote constants depending only on $p$. We may use the same symbol to denote different constants in different occurrences.

Let $K>0$ be larger than some absolute constant. We write
\[ 
    c(K), \ C(K), \ C'(K),
\]
to denote constants depending only on $K$. Once again, the same symbol could be used to denote different occurrences.

At the end of Section \ref{sec:gps}, we choose $K$ to be an absolute constant, taken sufficiently large. Once we do so, $K$ and any instances of $C(K)$ become absolute constants $C$.

\section{Dyadic Squares and Intervals}\label{s:squares}
Here, \underline{dyadic squares} are $Q^0=[-1,1]^2\subset \R^2$ and all {closed} squares arising from $Q^0$ by successive bisection. More precisely, any dyadic square $Q$ is either $Q^0$ or of the form $[i 2^{-k},(i+1) 2^{-k}]\times [j 2^{-k}, (j+1) 2^{-k}]$, for some $k \in \N\cup\{0\}$ and $i,j\in \{-2^k,\dots,2^k -1\}$. Similarly, \underline{dyadic intervals} consist of the interval $I^0=[-1,1]\subset \R$ and all the {closed} intervals {$[i 2^{-k}, (i+1)2^{-k}]$ for $i\in \{-2^k,\dots, 2^k-1\}$ and $k\in \N$}, obtained by successive bisection of $I^0$. Note that dyadic squares and intervals are \emph{closed}.

If $Q$ is any dyadic square, we denote its side length by $\delta_Q$ and if $I$ is a dyadic interval, we denote its length by $|I|$. Moreover, we denote the area $\delta_Q^2$ of a dyadic square $Q$ by $|Q|$. Each dyadic square $Q$ of side length $\delta_Q\leq 1$ arises by bisecting a dyadic square $Q^+ \supset Q$ of side length $2\delta_Q$; we call $Q^+$ the \underline{dyadic parent} of $Q$. Similarly, each dyadic interval $I$ of length $|I|\leq 1$ arises by bisecting a dyadic interval $I^+\supset I$ of length $2|I|$; we call $I^+$ the \underline{dyadic parent} of $I$.

If $I$ is a dyadic interval and $C>1$, then $CI$ denotes the {\emph{open}} interval of length $C|I|$ whose center coincides with that of $I$. Similarly, if $Q$ is a dyadic square and $C>1/4$ ($C\neq 1$), then $CQ$ denotes the {\emph{open}} square of side length $C\delta_Q$ whose center coincides with that of $Q$.

Throughout this article, we will make frequent use of the square $Q_{\inner} := [-2^{-10},2^{-10}]^2\subset Q^0$.

\section{Some Elementary Inequalities}\label{sec:elementary}
Suppose $F \in L^{2,p}(\R^2)$. In this section, we let $Q$ denote a \emph{closed} square in $\R^2$. We will henceforth simply refer to such $Q$ as a square. 
% \note{A: Charlie said to add a sentence about dyadic squares and intervals being closed in this section but everything in this section is stated for a general square so maybe something like this is better? And we can instead recall that dyadic squares/intervals are closed in the next section?} {\color{blue} M: Did the section numbers change? He says to add a remark in *Section 5* above which makes more sense. FWIW, writing $I \times J$ and then having $|I|=|J|$ seems distracting - in particular, in the Sobolev inequalities, the constant $C_p$ depends on the geometry on the domain - so if the aspect ratio ($|I|/|J|$) changed, then the constant would change. But this isn't the direction we're going with out work, so just making it a remark about cubes without labeling their sides is cleaner IMO. I recommend just writing: In this section, we let $Q$ denote a \emph{closed} square in $\R^2$. }
For any such $Q \subset \R^2$, we have the \\
\underline{Sobolev Inequality} {\cite[Lemma 7.16]{gt}}: $F$ may be expressed in the form $F_1 + L$ on $Q$, where $L$ is a first degree polynomial and $F_1$ is continuous on $Q$; moreover,
\[
\delta_Q^{2-2p} \max_{z \in Q}\{ |F_1(z)| \}^p + \delta_Q^{-p} \int_{z \in Q} |\nabla F_1(z)|^p dz \leq C_p  \left(\frac{1}{|Q|} \int_{z \in Q} |\nabla^2 F |^{\tilde{p}} dz \right)^{p/\tilde{p}}\delta_Q^2.
\]
For instance, given any point $z_0\in Q$, one may take $L(z) = F(z_0) + \frac{1}{|Q|}\int_{w \in Q} \nabla F(w)dw \cdot (z-z_0)$. This inequality is not sharp, but it's enough for our purposes.

We introduce the (non-centered) \emph{maximal function}
\[
\Mcal \varphi(z) = \sup_{Q \ni {z}} \left( \frac{1}{|Q|} \int_{w \in Q} |\varphi(w)|^{\tilde{p}} dw \right)^{1/\tilde{p}}
\]
for functions $\varphi: \R^2 \to \R$. Here the supremum is taken over all squares $Q$ containing $z$.

We recall the \\
\underline{Hardy-Littlewood Maximal Theorem} \cite[Chapter I, \S 3, Theorem 1]{Stein-HA}:
\[
\int_{\R^2} (\Mcal \varphi)^p dz \leq C_p \int_{\R^2} |\varphi|^p dz, \qquad {1<p\leq \infty, \ \varphi\in L^p(\R^2)}
\]
For $F \in L^{2,p}(\R^2)$ and a square $Q\subset \R^2$, we write $\Av_Q \nabla F$ to denote the mean $\frac{1}{|Q|} \int_Q \nabla F(z) dz$. Similarly, we write $\Av_Q \p_x F$ and $\Av_Q \p_y F$ to denote the components of $\Av_Q \nabla F$.

Applying the Sobolev Inequality to the open square $CQ$ ($C>1.1$) defined as in Section \ref{s:squares}, we obtain the following corollaries:

\begin{cor}\label{cor:t1}
    Let $\uz \in CQ$, and let $L_Q(z) = F(\uz) +(\Av_Q \nabla F) \cdot (z - \uz)$. Then for $|\al| \leq 1$, we have
    \begin{align*}
    \delta_Q^{|\al|p - 2p} \int_{1.1Q} |\p^{\al} (F-L_Q)|^p dz &\leq C_p \delta_Q^2 \left( \frac{1}{|CQ|} \int_{CQ} |\nabla^2 F|^{\tilde{p}} dz\right)^{p/\tilde{p}} \\
    &\leq C_p' \int_Q \left( \Mcal(|\nabla^2 F|) \right)^{p} dz.
    \end{align*}
\end{cor}

\begin{remark}\label{rem:t1}
    In fact, if $Q, Q'$ are squares such that $\delta_Q = \delta_{Q'}$ and $100Q \cap 100Q' \neq \emptyset$, then we have $C_p \delta_Q^2 \left( \frac{1}{|CQ|} \int_{CQ} |\nabla^2 F|^{\tilde{p}} dz\right)^{p/\tilde{p}} \leq C_p' \int_{Q'} \left( \Mcal(|\nabla^2 F|) \right)^{p} dz$.
\end{remark}

\begin{cor}\label{cor:t2}
    Let $z_1 = (x_1,0)$ and $z_2 = (x_2, 0)$ belong to $CQ$, with $|x_1-x_2|\geq c \delta_Q$. Then
    \begin{align*}
        \delta_Q^{2-p} \left| \Av_Q \p_x F - \frac{F(z_2) - F(z_1)}{x_2-x_1} \right|^p &\leq C_p \delta_Q^2 \left( \frac{1}{|CQ|} \int_{CQ} |\nabla^2 F|^{\pt} dz \right)^{p/\pt} \\
        & \leq C_p' \int_{Q} \left( \Mcal(|\nabla^2 F|) \right)^p dz.
    \end{align*}
\end{cor}

\begin{cor}\label{cor:t3}
    Let $z_1 = (x_1, y_1)$ and $z_2 = (x_1,y_2)$ belong to $CQ$, with $|y_1-y_2|\geq c \delta_Q$. Then
    \begin{align*}
        \delta_Q^{2-{p}}\left|\Av_Q {\p_y} F - \frac{F(z_2) - F(z_1)}{y_2 - y_1} \right|^p &\leq C_p \delta_Q^2 \left( \frac{1}{|CQ|} \int_{CQ} |\nabla^2 F|^{\pt} dz \right)^{p/\pt} \\
        & \leq C_p' \int_Q \left( \Mcal(|\nabla^2 F|) \right)^p dz.
    \end{align*}
\end{cor}

\begin{cor}\label{cor:t4}
    Let $Q_1, Q_2, Q$ be squares, and suppose $Q_1,Q_2 \subset CQ$ and $\delta_{Q_i} \geq c\delta_Q$ for $i=1,2$. Then
    \begin{align*}
        \delta_Q^{2-p}\left|\Av_{Q_1} \nabla F - \Av_{Q_2} \nabla F \right|^p &\leq C_p \delta_Q^2 \left( \frac{1}{|CQ|} \int_{CQ} |\nabla^2 F|^{\pt} dz \right)^{p/\pt} \\
        & \leq C_p' \int_Q \left( \Mcal(|\nabla^2 F|) \right)^p dz.
    \end{align*}
\end{cor}

We use also the following estimates:

\begin{lem}\label{lem:t10}
    Let $Q,Q'$ be squares. Suppose $Q\cap Q'\neq \emptyset$, and suppose that their side lengths differ by at most a factor of two. Let $\uz \in CQ$. Let $P$ be a first degree polynomial on $\R^2$. Then the quantities $\sum_{|\al| \leq 1} \delta_Q^{|\al|p - 2p} \int_{1.1Q \cap 1.1 Q'} |\p^\al P |^p dz$ and $\delta_Q^{2} \left(|\delta_Q^{-2}P(\uz)|^p + |\delta_Q^{-1} \p_x P|^p + |\delta_Q^{-1} \p_y P|^p \right)$ differ by at most a factor $C_p$, where $C_p$ is determined by $C$ and $p$. 
\end{lem}
The elementary proof is left to the reader.

\begin{lem}\label{lem:s1}
    Let $x_0, x_1, ..., {x_{L-1}} \in \R$, and let $1<p<2$. Then
\begin{align}
    \sum_{\ell = 0}^{L-1} \left| \sum_{k = \ell}^{L-1} x_k \right|^p 2^{(2-p)\ell} \leq C_p \sum_{\ell = 0}^{L-1} |x_\ell |^p 2^{(2-p)\ell}.
\end{align}
\end{lem}

\begin{proof}
    % {\color{blue} M: I've corrected this proof and added detail, so it makes sense. but I'd appreciate if you removed what detail is excessive. }{\color{red}It looks good to me - I would be happy with leaving all the details in, for clarity.} 
    To see this, let $\beta = 2^{2/p-1-\eps/p} \in (1, 2^{2/p-1})$ for $\eps \in (0,2-p)$; note that 
\begin{align*}
    \left| \sum_{k = \ell}^{L-1} x_k \right|^p \leq \left( \sum_{k = \ell}^{L-1} |x_k|^p \beta^{(k- \ell)p} \right) \left( \sum_{k = \ell}^{L-1} \beta^{-(k- \ell)p'} \right)^{p/p'}
\end{align*}
by H\"older's inequality, where $p' = \frac{p}{p-1}$. Consequently,
\begin{align*}
    \sum_{\ell = 0}^{L-1} \left| \sum_{k = \ell}^{L-1} x_k \right|^p 2^{(2-p)\ell} &\leq C_p \sum_{\ell = 0}^{L-1} \left( \sum_{k = \ell}^{L-1} |x_k|^p \beta^{(k- \ell)p} \right) 2^{(2-p)\ell} \\
    &= C_p \sum_{k = 0}^{L-1} |x_k|^p \left( \sum_{\ell = 0}^k \beta^{(k- \ell)p}2^{(2-p)\ell} \right) \\
    &= C_p \sum_{k = 0}^{L-1} |x_k|^p \left( \sum_{\ell = 0}^k 2^{(k- \ell)(2-p-\eps)}2^{(2-p)\ell} \right),
\end{align*}
and thus
\begin{align*}
    \sum_{\ell = 0}^{L-1} \left| \sum_{k = \ell}^{L-1} x_k \right|^p 2^{(2-p)\ell} &\leq C_p \sum_{k = 0}^{L-1} |x_k|^p 2^{(2-p)k} \left( \sum_{\ell = 0}^k 2^{(k- \ell)(-\eps)} \right) \\
    &\leq C_p' \sum_{k = 0}^{L-1} |x_k |^p 2^{(2-p)k}.
\end{align*}
\end{proof}

\begin{lem}\label{lem:s2}
    Let $L$ be a first degree polynomial on $\R^2$, and let $\tilde{L}(x,y) = L(x,0)$. Suppose $Q \subset \R^2$ is a square of the form $I \times J$, with $0 \in CJ$. Then 
\begin{align}
    \int_Q |\tilde{L}(x,y)|^p dx dy \leq C_p \int_Q |L(x,y)|^pdxdy.
\end{align}
\end{lem}

We leave the proof of this inequality to the reader.

\section{The Set $E$}

Let $E = \ue \times \{0 \} \subset \R^2$, where $\ue \subset \R$ is a finite set. Let $N = \# E \geq 2$. We suppose $\max \ue = 2^{-11}$ and $\min \ue = -2^{-11}$. 

We will find it convenient to add the point $(-1,0)$ to $E$; we set $\ue^+ = \ue \cup \{-1\}$, and let $E^+ = \ue^{+} \times \{0\}$. 

We will work with functions $f: E \to \R$ and with functions $f^{+}: E^{+} \to \R$.

\section{The Calder\'on-Zygmund Decomposition}
Starting with the square $Q^0$, we perform repeated bisection \`a la Calder\'on-Zygmund, stopping at a square $Q$ provided $\# (E \cap 3Q) \leq 1$. Recall $3Q$ denotes an \emph{open} square of side length $3\delta_Q$. This process results in a decomposition of $Q^0$ consisting of finitely many \emph{closed} \underline{Calder\'on-Zygmund squares}. We write CZ to denote the set of Calder\'on-Zygmund squares. Equivalently, the CZ squares are precisely the maximal dyadic squares $Q$ such that $\#(E \cap 3Q) \leq 1$.

\begin{remark}\label{r:contact}
    All the $\CZ$ squares $Q$ that we encounter have one of the following forms:
    \begin{align*}
        &\text{``Contact Squares"} && I \times [0, |I|] \text{ or } I \times [-|I|, 0] \\
        &\text{``Contactless Squares"} &&I \times[|I|, 2|I|] \text{ or } I \times [-2|I|, -|I|],
    \end{align*}
    for a dyadic interval $I$. This follows from an obvious induction, together with the observation that $E \cap 3Q= \emptyset$ for contactless squares $Q$ (because $E \subset \R \times \{0\}$, while the open square $3Q$ is disjoint from $\R \times \{0\}$).
\end{remark}

%{\color{blue} For $Q \in CZ$, we have $\delta_Q \leq 1/2$ (If $\delta_Q \geq 1$, $E \subset 3Q$ and $\#E \geq 2$, a contradiction).}

We write $Q \lra Q'$ for CZ squares $Q, Q'$ to indicate that $Q\cap Q' \neq \emptyset$. If the CZ squares $Q, Q'$ {satisfy $Q\cap Q'\neq \emptyset$}, then $\frac{1}{2} \delta_Q \leq \delta_{Q'} \leq 2 \delta_Q$. We recall the standard proof: Suppose {$Q\cap Q'\neq\emptyset$} with $\delta_{Q'} \leq \frac{1}{4} \delta_Q$. Then the dyadic parent $(Q')^+$ of $Q'$ satisfies $3(Q')^+ \subset 3Q$. Since $\#(E \cap 3Q) \leq 1$, it follows that $\# (E \cap {3(Q')^+}) \leq 1$, and consequently we should not have bisected $(Q')^+$ to arrive at the CZ square $Q'$. This contradiction shows that $\delta_{Q'} \geq \frac{1}{2} \delta_Q$. By the same reasoning, $\delta_Q \geq \frac{1}{2}\delta_{Q'}$. In light of this ``good geometry" we recall a familiar bound:

{
%We recall the \\ %$\{Q_\nu\}_{\nu=1}^{\nu_{\max}}$ contained in $Q^0$, with pairwise disjoint interiors. 

\underline{Patching Estimate} \cite[Lemma 9.1]{arie3}, \cite[Lemma 5]{CI23}: Let $(\theta_Q)_{Q \in CZ} \subset C^{\infty}(Q^0)$ be a smooth partition of unity satisfying for $Q \in CZ$, $0 \leq \theta_Q \leq 1$; $\theta_Q$ vanishes on $Q^0 \setminus 1.1 Q$; $|\p^\alpha \theta_Q| \leq C \delta_{Q}^{-|\alpha|}$ for $|\alpha| \leq 2$; $\theta_Q|_{\frac{1}{2} Q} \equiv 1|_{\frac{1}{2}Q}$; and $\sum_{Q \in CZ} \theta_Q \equiv 1$ on $Q_\inner$. Suppose $G_Q\in L^{2,p}(1.1 Q)$ for each $Q \in CZ$ satisfying $1.1Q \cap Q_\inner \neq \emptyset$. We define the function $G: Q_\inner \to \R$ as $G(z): = \sum_{Q \in CZ} \theta_Q(z) G_Q(z)$. We have
\begin{align}
    \|G\|_{L^{2,p}(Q_{\inner})}^p \leq C_p \bigg(&\sum_{\substack{Q \in CZ: \\ 1.1Q \cap Q_{\inner} \neq \emptyset}} \|G_Q\|_{L^{2,p}(1.1 Q)}^p \nonumber \\
    &+ \sum_{|\alpha|\leq 1}\sum_{\substack{Q, Q'\in CZ: \; Q \leftrightarrow Q'\\ 1.1Q \cap Q_{\inner} \neq \emptyset, \;  1.1Q' \cap Q_{\inner} \neq \emptyset}} \delta_{Q}^{|\alpha|p-2p} \int_{1.1 Q \cap 1.1 Q'} |\partial^\alpha(G_Q - G_{Q'})|^p \, dz \bigg). \label{eq:patching}
\end{align}}

Let $\pi: \R^2 \to \R$ be the projection $(x,y) \mapsto x$. The \underline{shadow} of a dyadic square $Q$ is the dyadic interval $\pi (Q)$. 

We make the following observation regarding shadows:

\begin{lem}\label{l:parent-shadow}
    Let $I = \pi (Q)$ for a CZ square $Q$. If $|I| \leq \frac{1}{4}$, then $I^+ = \pi (Q')$ for a CZ square $Q'$.
\end{lem}

\begin{proof}
    We note that $Q$ arises by bisecting its dyadic parent $Q^+$, which in turn arises by bisecting $Q^{++}$. Both $Q^+$ and $Q^{++}$ are contact squares, else we would not have bisected them. Then one of the two contactless squares $Q'$ obtained by bisecting $Q^{++}$ belongs to CZ and satisfies $\pi (Q') = I^+$. See Figure \ref{fig:test} for a depiction of this.
\end{proof}

\begin{figure}[h]
\begin{subfigure}{.5\textwidth}
\centering
\includegraphics[width=0.7\textwidth]{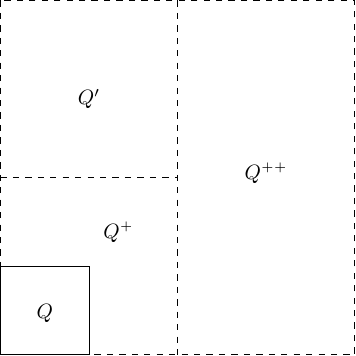}
\end{subfigure}
\begin{subfigure}{.5\textwidth}
\centering\includegraphics[width=0.7\textwidth]{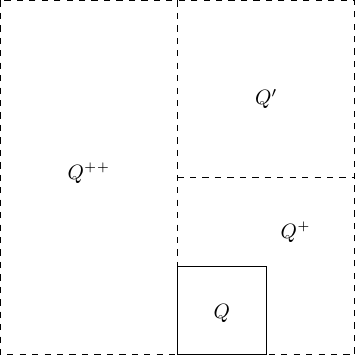}
\end{subfigure}
\caption{Two possible scenarios for a {contactless} child $Q'$ of $Q^{++}$.} \label{fig:test}
\end{figure}

The square $Q^0$ satisfies $\#(E \cap 3Q^0) \geq 2$ since $(-2^{-11},0), (2^{-11},0) \in E$. Hence every CZ square $Q$ arises by bisecting $Q^+$, so $\#(E \cap 3Q^+) \geq 2$, else we wouldn't have bisected $Q^+$. This means that $\#(\ue \cap {3I^+}) \geq 2$, where ${I^+ = \pi (Q^+)}$. 

For each shadow $I$, we fix a point $\hat{x}(I) \in \ue \cap 3I^+ \subset \ue \cap 7I$. 

Let $Q$ be a CZ square. We say that $Q$ is \rel{} if $1.1Q \cap Q_{\inner} \neq \emptyset$, {see Section \ref{s:squares},} and we denote the set of \rel{} CZ squares by $\CZ_{\text{rel}}$

We note that every \rel{} CZ square $Q$ satisfies $\delta_Q \leq 2^{-8}$. Indeed if $\delta_Q \geq 2^{-7}$ and $1.1Q \cap Q_{\inner} \neq \emptyset$, then $3Q \supset Q_{\inner}$. Hence $\# (E \cap 3Q) \geq 2$, so $Q$ cannot be a CZ square.

A \underline{\rel{} shadow} is defined to be the shadow of a \rel{} CZ square. \rel{} shadows $I$ thus satisfy ${1.1}I \cap [-2^{-10}, 2^{-10}] \neq \emptyset$ and $|I| \leq 2^{-8}$.

\section{Groups of Shadows and Squares} \label{sec:gps}

Let $I = \pi (Q)$ be the shadow of $Q \in CZ$.

{Recall, from Section \ref{sec:constants} we assume $K>0$ larger than a universal constant.} We say that $I$ is \easy{} if 
\begin{equation}\label{e:k-easy}
    \diam(\ue^+ \cap KI) \geq K^{-1} |I|.
\end{equation}
Otherwise, we say that $I$ is \hard{}.

Fix $x^* \in \ue^+$. Then we define $\Gp(x^*)$ (the ``group" of $x^*$) to be the set of all shadows $I$ such that 
\begin{itemize}
    \item[(a)] $x^* \leq \inf 41I$, and
    \item[(b)] $\ue^+\cap (x^*, \inf 9I] = \emptyset.$
\end{itemize}

The following elementary lemma will be useful when decomposing shadows into disjoint subfamilies.
\begin{lem}\label{l:relevant-hard}
    {Let $I$ be a shadow. Then $I\in\Gp(x^*)$ for at most one $x^*\in\ue^+$. If in addition $I$ is a \hard{} \rel{} shadow, then provided $K\geq 41$, there exists $x^* \in \ue^+$ such that $I\in\Gp(x^*)$.}
\end{lem}

\begin{proof}
{Let $I$ be a shadow; and suppose} (a) and (b) hold for $x^* \in \ue^+$. Because $x^*$ is the maximal element of $(-\infty, \inf 9I] \cap \ue^+$, (a) and (b) cannot also hold for another element of $\ue^+$. %Consequently, any given shadow $I$ belongs to $\Gp(x^*)$ for at most one $x^* \in \ue^+$.

Now let $I$ be a \hard{} \rel{} shadow. Then since $1.1 I \cap [-2^{-10}, 2^{-10}] \neq \emptyset$ and $|I| \leq 2^{-8}$, we have $-1 < \inf 9I$, hence $\tilde{x} < \inf 9I$ for some $\tilde{x} \in \ue^+$ (we added $-1$ to $\ue$ to form $\ue^+$ for this reason). Let
\[
    x^* := \max\{\tilde x\in \ue^+ : {\tilde x \leq \inf 9 I} \}.
\]
We have two possibilities; either $x^* \in 41 I\setminus 9I$ or $x^* \leq \inf 41 I$. Suppose $x^* \in 41I{\setminus 9I}$ and recall that there exists $\hat{x}(I) \in \ue\cap 7I$. Thus, $x^*, \hat{x}(I) \in \ue^+ \cap 41 I$ and $|x^* - \hat{x}(I)| \geq |I|$, implying $I$ is \easy{}. This contradicts the assumption $I$ is \hard{}, so we must have $x^* \leq \inf 41 I$. {By our choice of $x^*$, property (b) also holds. Thus, $I \in \Gp(x^*)$.}
\end{proof}

Until further notice, we fix a point $x^* \in \ue^+$ and study its group $\Gp(x^*)$. We assume that $\Gp(x^*) \neq \emptyset$.

\begin{lem} \label{lem:9.1}
    Let $J, J' \in \Gp(x^*)$, then $9J \cap 9J' \neq \emptyset$.
\end{lem}

\begin{proof}
    Suppose not. Without loss of generality, we may suppose that $9J$ lives to the left of $9J'$, or rigorously:
    \[
        x^* \leq \inf 41J < \inf 9J < \sup 9J \leq \inf 9J'.
    \]
    Since $J$ is a shadow, $\hat{x}(J) \in \ue \cap 7J$. In particular, ${\hat{x}(J)\in\ue^+\cap(x^*, \inf 9J']}$, contradicting condition (b) for $J'$.
\end{proof}

Fix a shadow $I_0 \in \Gp(x^*)$ with 
\[
    |I_0| = \min\{|I| : I\in \Gp(x^*)\}.
\]

\begin{lem} \label{lem:9.2}
    {Let $c>0$ and suppose that} $J \in \Gp(x^*)$ of length $|J|>c \dist(x^*, I_0)$. {Then $J$ is \easy{} for $K > \frac{2}{c}+19>0$.}
\end{lem}

\begin{proof}
    {Fix $c>0$ and $J$ as in the statement of the lemma.} We have $9J \cap 9I_0 \neq \emptyset $ by Lemma \ref{lem:9.1}, and $|J| \geq |I_0|$ by the definition of $I_0$. Since also $|J|>c \dist(x^*,I_0)$, it follows that {$x^* \in K J$ for any $K>\frac{2}{c}+19$.} Thus, $\hat{x}(J)$ and $x^*$ both belong to $\ue^+ \cap KJ$. Moreover {property (a) guarantees that} $x^* \leq \inf 41 J$, whereas $\hat{x}(J) \in 7J$. Therefore $|\hat{x}(J) - x^*| \geq |J|$. So $\diam(KJ \cap \ue^+) \geq K^{-1}|J|$, proving that $J$ is \easy{}.
\end{proof}

{Now consider the family $I_0 \subset I_1 \subset ... \subset I_L$}, where each $I_\ell$ ($\ell \geq 1$) is the dyadic parent of $I_{\ell-1}$, and $L$ is the largest integer for which all the $I_\ell$ ($\ell \leq L$) belong to $\Gp(x^*)$. {We first demonstrate that $I_L$ satisfies the hypothesis of Lemma \ref{lem:9.2} with $c = \frac{1}{43}$.

\begin{lem} \label{lem:9.3}
    The interval $I_L$ satisfies $|I_L|\geq \frac{1}{42} \dist(x_*, I_0)$.
\end{lem}

\begin{proof}
If $|I_L|>1/4$, then we must have $x^* \in 7I_L$ because $[-1,1] \subset 7I_L$, a contradiction of property (a) above. Thus $|I_L| \leq 1/4$. Then by Lemma \ref{l:parent-shadow}, $I_L^+$ is a shadow. Since $I_L \in \Gp(x^*)$, we have $x^* \leq \inf 41 I_L$. Suppose that $x^* \leq \inf 41 I_L^+$. Since $\inf 9(I_L^+) < \inf 9I_0$ and $\ue^+\cap (x^*, \inf 9 I_0] = \emptyset$, it follows that $\ue^+\cap (x^*, \inf 9 I_L^+] = \emptyset$. Thus, $I_L^+$ satisfies conditions (a) and (b) above, i.e., ${I_L^+} \in \Gp(x^*)$, contradicting the maximality of the index $L$. Thus we must have $x^*\in 41 I_L^+$, and
\[
    \dist(x^*, I_0) \leq 20|I_L^+| + |I_L^+| \leq 42|I_L|.
\]

%\note{A: hmm I am not sure I follow: surely since $I_L\in \Gp(x^*)$, we cannot have $x^*\in 7I_L$ because $x^*\leq \inf 41 I$? But actually, can't we say that $|I_L|\leq 1/4$ precisely for this reason? Because otherwise there is no way we can ensure that $x^*\leq \inf 41 I$ for any $x^*\in \ue^+$. I'm adding a justification at the beginning of the proof in red; {\color{red}First of all, observe that $|I_L|\leq 1/4$, since otherwise $\dist((-1,0), L) \leq \frac{7}{4}$ which would in turn imply that $x^* \geq 15 I_L$, therefore failing property (a) above.} {M:I think you're correct - I'm just implementing your correction efficiently; apologies for the error here. I've commented this out we can see it.}
\end{proof}

{We will henceforth fix $c=\frac{1}{43}$ in Lemma \ref{lem:9.2}. Then, by Lemma \ref{lem:9.2}, if $J \in \Gp(x^*)$ is \hard{}, we must have $|J| \leq \frac{1}{42} \dist(x^*, I_0)<|I_L|$.} Consequently, for some $\ell$ ($0 \leq \ell \leq {L-1}$), we have $|J| = |I_\ell|$. Since $J$ and $I_\ell$ belong to $\Gp(x^*)$, Lemma \ref{lem:9.1} tells us that $9J \cap 9I_\ell \neq \emptyset$.

We now pass from groups of shadows to groups of CZ squares.
\begin{define}
    Let $z^* = (x^*,0) \in E^+$. We define 
    \[
        \GP(z^*) := \{ Q\in \CZ : \pi(Q) \in \Gp(x^*)\},
    \]
    namely $\GP(z^*)$ consists of all CZ squares whose shadows belong to $\Gp(x^*)$. Corresponding to the sequence of shadows $I_0 \subset I_1 \subset\dots \subset I_L$ defined as above, we introduce a sequence of CZ squares $Q_0, Q_1, ..., Q_L$ by choosing $Q_\ell$ {such that $\pi(Q_\ell) = I_\ell$ for $\ell \in 0,..,L$.} For each $\ell$ we have at least one and at most four possible choices for $Q_\ell$; see Remark \ref{r:contact} above.
\end{define}

Collecting the above results regarding the properties of $\Gp(x^*)$, and noting Lemma \ref{l:parent-shadow}, we obtain the following:

\begin{lem} \label{lem:9.4}
    Let $z^* \in E^+$ with $\GP(z^*) \neq \emptyset$. Then there exist CZ squares $Q_0, Q_1, ..., Q_L \in \GP(z^*)$ {and a constant $C>0$} with the following properties:
    \begin{itemize}
        \item[(i)] $\delta_{Q_\ell} = 2^\ell \delta_{Q_0} \text{ for } 0 \leq \ell \leq L$,
        \item[(ii)] $Q_\ell \subset CQ_{\ell'} \text{ for } \ell < \ell'$,
        \item[(iii)] $Q_L$ is \easy{},
        \item[(iv)] given any \hard{} $Q \in \GP(z^*)$, there exists $\ell \; (0 \leq \ell \leq {L-1})$, with $\delta_Q =\delta_{Q_\ell} \text{ and } 100 Q \cap 100 Q_\ell \neq \emptyset$.
    \end{itemize}
\end{lem}

The squares $Q_0, Q_1, ..., Q_L$ and the number $L$ in Lemma \ref{lem:9.4} depend on $z^*$; {we will thus sometimes write $Q_\ell(z^*)$ to denote the square $Q_\ell$ and emphasize the dependence on $z^*$}. We write $\Qsmall(z^*)$ to denote $Q_0$ and $\Qbig(z^*)$ to denote $Q_L$. We let 
\[
    \Gang(\ell, z^*):=  \{Q \in \GP(z^*) : \delta_Q = \delta_{Q_\ell}, \ 100Q \cap 100 Q_\ell \neq \emptyset\}.
\]
The CZ squares $Q_\ell(z^*)$ may or may not be \hard{}. We write $L(z^*)$ to denote the number $L$ associated to $z^*$ in Lemma \ref{lem:9.4}.

\begin{lem}\label{lem:9.5}
    Let $Q \in \GP(z^*)$ and $\tilde{Q} \in \GP(\tilde{z}^*)$. If $Q \lra \tilde{Q}$, then $z^* = \tilde{z}^*$.
\end{lem}

\begin{proof}
        Suppose that the shadows $J,\tilde{J}$ satisfy $J \in \Gp(x^*)$ and $\tilde{J} \in \Gp(\tilde{x}^*)$ for two points $x^*, \tilde x^* \in \ue^+$. Suppose that $J \cap \tilde{J} \neq \emptyset$, and that $\frac{1}{2}|J| \leq |\tilde{J}| \leq 2 |J|$. Then we claim $x^* = \tilde{x}^*$, and the Lemma follows immediately as the shadows $\pi(Q)$ and $\pi(\tilde{Q})$ satisfy these conditions.
        
        To see the claim is true, suppose without loss of generality that $\inf 9J \leq \inf 9\tilde{J}$. We know that $\tilde{x}^* \leq \inf 41 \tilde{J}$, and our assumptions on $J, \tilde{J}$ imply that $\inf 41 \tilde{J} < \inf 9J$. Hence, $\tilde{x}^* < \inf 9J$. Moreover, by the defining property (b) for $\tilde J$ we have $\ue^+\cap(\tilde{x}^*, \inf 9\tilde{J}]=\emptyset$, and therefore $\ue^+ \cap (\tilde{x}^*, \inf 9J] = \emptyset$ also. Thus $\tilde{x}^*$ is the maximal point of $\ue^+ \cap (-\infty, \inf 9J]$. However, since $J \in \Gp(x^*)$, we know that $x^*$ is the maximal point of $\ue^+ \cap (-\infty, \inf 9J]$. We conclude that $x^* = \tilde{x}^*$.
\end{proof}

We conclude this section by estimating the number of \easy{} CZ squares. Recall that $N = \# E$.

\begin{lem} \label{lem:9.6}
    The number of \easy{} CZ squares is at most $C(K) \cdot N$. Moreover, the number of CZ squares $Q$ such that $1.1Q \cap E \neq \emptyset$ is at most $C N$.
\end{lem}

\begin{proof}
    For each $z \in E$, there are at most $C$ distinct CZ squares $Q$ for which $1.1Q \ni z$. Summing over $z \in E$, we see at once that the number of CZ squares $Q$ such that $1.1 Q \cap E \neq \emptyset$ is at most $CN$.

    To estimate the number of \easy{} CZ squares, we use the Well-Separated Pairs Decomposition, due to Callahan and Kosaraju \cite{ck}. As a consequence of the main result of \cite{ck}, there exist points $\{(x_\nu',x_\nu'')\}_{\nu=1}^{\nu_{\max}} \subset \ue^+ \times \ue^+ \setminus \text{Diagonal}$ with the following properties:
    \begin{enumerate}[(i)]
        \item\label{eq:wspd1} Given $(x', x'') \in \ue^+ \times \ue^+ \setminus \text{Diagonal}$, there exists $\nu$ with $|x' - x_\nu'| + |x'' - x_\nu''| < \frac{1}{100} |x'-x''|$
        \item\label{eq:wspd2} $\nu_{\max} \leq CN$.
    \end{enumerate}
    Now let $Q$ be an \easy{} CZ square with shadow $I = \pi (Q)$. By definition, there exist $x', x'' \in {\ue^+} \cap KI$ such that $|x'-x''| \geq K^{-1}|I|$. For such $x',x''$, we pick $\nu$ as in (\ref{eq:wspd1}). Thus
    \begin{align}
        |x_\nu' - x_\nu''| \geq \frac{1}{2}K^{-1}|I| \text{ and } x_\nu', x_\nu'' \in 2KI. \label{eq:wspd3}
    \end{align}

For fixed $(x_\nu', x_\nu'')$, there are most $C(K)$ dyadic intervals $I$ satisfying (\ref{eq:wspd3}). Summing over $\nu$ and applying (\ref{eq:wspd2}), we find that there are at most $C(K) N$ distinct shadows of \easy{} CZ squares. Since each shadow arises from at most four CZ squares, we conclude there are at most $C(K) N$ \easy{} CZ squares.

\end{proof}

We now take $K$ to be an absolute constant, large enough that the preceding arguments work. The constant $C(K)$ in Lemma \ref{lem:9.6} is therefore also an absolute constant, henceforth denoted by $C$. Moreover, we now write that an \easy{} shadow $I$ satisfies $\diam(\ue^+ \cap CI) \geq c|I|$, and an \easy{} CZ square $Q$ satisfies $\diam(E^+ \cap CQ) \geq c\delta_Q$.

\section{Assigning special points to CZ squares}\label{sec:special pts}
Let $Q$ be a \rel{} CZ square. We will associate to $Q$ three points of $E^+$, denoted $z_1 (Q), z_2 (Q), \uz(Q)$, as follows.

Recall {from Lemma \ref{l:relevant-hard}} that a \rel{} $\CZ$ square is either \easy{}, or else it belongs to $\GP(z^*)$ for a unique $z^* \in E^+$. Using this dichotomy, we proceed to define $z_1 (Q), z_2 (Q) \in E^+$ for any \rel{} $\CZ$ square $Q$.
\begin{enumerate}[(a)]
    \item\label{enum:easy z_1} If $Q$ is \easy{}, then {by the defining property \eqref{e:k-easy},} there exist two points $z_1, z_2 \in E^+ \cap {CQ}$ such that $|z_1 - z_2| \ge {c \delta_Q}$. Define $z_1 (Q) := z_1$ and $z_2 (Q) := z_2$.

    \item If $Q$ is \hard{} and belongs to $\GP(z^*)$ for some $z^* \in E^+$, then recall {from Lemma \ref{lem:9.4}} that the $\CZ$ square $\Qbig (z^*)$ is \easy{}. Hence, $z_1 (\Qbig (z^*))$ and $z_2 (\Qbig (z^*))$ have already been defined in case \eqref{enum:easy z_1}. Define $z_1 (Q) := z_1 (\Qbig (z^*))$ and $z_2 (Q) := z_2 (\Qbig (z^*))$. Note that since $\Qbig (z^*)$ is \easy{}, $|z_1(Q) - z_2(Q)| \geq c\delta_{\Qbig (z^*)}$.
\end{enumerate}

Now, we define the points $\uz(Q) \in E$. We proceed by cases. Let $Q$ be a \rel{} $\CZ$ square.

\begin{enumerate}[(1)]
    \item\label{enum:uz1} If $1.1Q \cap E \neq \emptyset$, then because $Q$ is a CZ square, there is a unique point in $1.1Q \cap E$. Let $\uz(Q)$ be this unique point.

    \item\label{enum:uz2} If $1.1Q \cap E = \emptyset$ and $Q$ is \easy{}, then we take $\uz(Q)$ to be any point of $E \cap 7Q$ (this set contains $E\cap 3Q^+$ and $\# (E \cap 3Q^+) \ge 2$) 
    
    \item\label{enum:uz3} If $1.1Q \cap E = \emptyset$ and $Q$ is \hard{} with $Q \in \GP(z^*)$ 
    {for some $z^*\in E^+$}, let $\uz(Q) = \uz(z^*)$, where $\uz(z^*) \in E \cap 7\Qsmall (z^*)$ is fixed for each $z^* \in E^+$. Thus $\uz(Q)$ only depends on $z^*$ and not the particular choice of $Q$.
\end{enumerate}

We have thus defined $z_1 (Q), z_2 (Q), \uz(Q)$ for all \rel{} CZ squares. We make the following simple observations.

\begin{lem}\label{lem:zQ in CQ}
Let $Q$ be a \rel{} CZ square. Then:
\begin{enumerate}[(a)]
    \item $\uz(Q) \in E$ and if $1.1Q \cap E \neq \emptyset$, then $\uz(Q)$ is the unique point in $1.1Q \cap E$.
    \item $\uz(Q) \in CQ$.
\end{enumerate}
\end{lem}

\begin{proof}
   Part (a) is clear, so we tackle part (b).
   This is clear in cases \eqref{enum:uz1} and \eqref{enum:uz2}. Suppose we are in case \eqref{enum:uz3}, so $1.1Q \cap E = \emptyset$ and $Q$ is \hard{}, so $Q \in \GP(z^*)$ for some $z^* \in {E^+}$. Let $Q_0, Q_1, \cdots, Q_L$ be as in Lemma \ref{lem:9.4} {for this choice of $z^*$}. Then $\Qsmall (z^*) = Q_0$, and for some $0 \le \ell \le L$ we have $\delta_Q = \delta_{Q_\ell}$ and $100Q \cap 100Q_\ell \neq \emptyset$. By that same lemma, we have $Q_0 \subset CQ_\ell$, hence $7Q_0 \subset 7CQ_\ell$. Since $\delta_Q = \delta_{Q_\ell}$ and $100Q \cap 100Q_\ell \neq \emptyset$, we have $7CQ_\ell \subset C' Q$. Thus, $\uz(Q) \in 7\Qsmall (z^*) = 7Q_0 \subset C' Q$, completing the proof of the lemma.
\end{proof}

\section{The Norm of an Interpolant}

Let $f^+ : E^+ \to \R$, and let $F \in L^{2,p} (\R^2)$ satisfy $F = f^+$ on $E^+$. We will derive a lower bound on $\Vert F \Vert_{L^{2,p} (\R^2)}$.

Let $Q \in \CZ$ be \rel{} and \easy{}.

Recall that $z_i (Q) = (x_i (Q), 0)$ for $i = 1, 2$ and $\uz(Q) = (\ux(Q), 0)$ belong to $E^+ \cap CQ$ and satisfy $|z_1(Q)-z_2(Q)| \geq c \delta_Q$. 

Let
\begin{equation}\label{eqn:norm0}
    \hat{L}_Q(x, y)=f^{+}(\underline{z}(Q))+ {\left[\frac{f^+ \left(z_2(Q)\right)-f^+ \left(z_1(Q)\right)}{x_2(Q)-x_1(Q)}\right] \cdot(x-\underline{x}(Q)) } +\left[\Av_Q \partial_y F\right] \cdot y.
\end{equation}

Applying Corollaries \ref{cor:t1} and \ref{cor:t2} from Section \ref{sec:elementary}, we see that for all $Q \in \CZ$ which are both \rel{} and \easy{},
\begin{equation}\label{eqn:norm1}
    \sum_{|\alpha| \le 1} \delta_Q^{|\alpha| p - 2p} \int_{1.1Q} |\partial^\alpha (F - \hat{L}_Q)|^p \, dz \le C_p \int_Q [ \Mcal(|\nabla^2 F|)]^p \, dz.
\end{equation}

We want an analogue of estimate \eqref{eqn:norm1} when $Q \in \CZ$ is \rel{} and \hard{}. We won't be able to derive such an estimate for an individual $Q$. Instead, we consider together all the \hard{} $Q$ in $\GP(z^*)$ for a given $z^* \in E^+$.

Fix $z^* \in E^+$ with $\GP(z^*) \neq \emptyset$, and let $Q_0, Q_1, \cdots, Q_L$ be as in Lemma \ref{lem:9.4}. By definition, $Q_0 = \Qsmall (z_*)$ and $Q_L = \Qbig (z^*)$, and recall that $Q_L$ is \easy{} {by Lemma \ref{lem:9.4}(iii)}. Since $Q_{\ell-1} \subset CQ_\ell$ for $\ell = 1, \cdots, L$, we see from Corollary \ref{cor:t4} in Section \ref{sec:elementary} that
\begin{equation*}
    |\Av_{Q_{\ell-1}} \partial_x F - \Av_{Q_\ell} \partial_x F|^p \cdot \delta_{Q_\ell}^{2-p} \le C_p \int_{Q_\ell} [\Mcal(|\nabla^2 F|)]^p \, dz.
\end{equation*}
Thus, by Lemma \ref{lem:s1} with $x_\ell = {\Av_{Q_{\ell}} \partial_x F - \Av_{Q_{\ell+1}} \partial_x F}$, we see that 
% {\color{red} RHS indices changed to $\ell=0$ through $L-1$ throughout the argument}{\color{blue}A: I think on the right we want $\ell=1$ through $L$, no? Just so that it is consistent with the inequality above. I think you're correct, Anna. Probably worth explaining in the referee response (and re-correcting) -M  }{Ok, I'll change it back and add it to that}
\begin{equation}\label{eqn:norm2}
    \sum_{\ell=0}^{{L-1}} |\Av_{Q_{\ell}} \partial_x F - \Av_{Q_L} \partial_x F|^p \cdot \delta_{Q_\ell}^{2-p} \le C_p \sum_{{\ell=1}}^{L} \int_{Q_\ell} [\Mcal(|\nabla^2 F|)]^p \, dz.
\end{equation}
Recall that $\Gang(\ell, z^*)$ consists of dyadic squares $Q$ such that $\delta_Q = \delta_{Q_\ell}$ and $100Q \cap 100Q_\ell \neq \emptyset$. Since $\# \Gang(\ell, z^*) \le C$, it follows from \eqref{eqn:norm2} that
\begin{equation}\label{eqn:norm3}
    \sum_{\ell=0}^{{L-1}} \sum_{Q \in \Gang(\ell, z^*)} |\Av_{Q_{\ell}} \partial_x F - \Av_{Q_L} \partial_x F|^p \cdot \delta_{Q_\ell}^{2-p} \le C_p \sum_{{\ell=1}}^{L} \int_{Q_\ell} [\Mcal(|\nabla^2 F|)]^p \, dz.
\end{equation}
Also, by Corollary \ref{cor:t4}, we have for all $Q \in \Gang(\ell, z^*)$,
\begin{equation*}
    |\Av_{Q} \partial_x F - \Av_{Q_\ell} \partial_x F|^p \cdot \delta_{Q}^{2-p} \le C_p \int_{Q_\ell} [\Mcal(|\nabla^2 F|)]^p \, dz.
\end{equation*}
Summing over $Q$ and again recalling that $\# \Gang(\ell, z^*) \le C$, for each {$\ell = 0,\dots, L-1$} we obtain
\begin{equation*}
    \sum_{Q \in \Gang(\ell, z^*)} |\Av_{Q} \partial_x F - \Av_{Q_\ell} \partial_x F|^p \cdot \delta_{Q}^{2-p} \le C_p \int_{Q_\ell} [\Mcal(|\nabla^2 F|)]^p \, dz.
\end{equation*}
Summing this over $\ell$ and combining with \eqref{eqn:norm3} via the triangle inequality, we get
\begin{equation}\label{eqn:norm4}
    \sum_{\ell=0}^{L-1} \sum_{Q \in \Gang(\ell, z^*)} |\Av_{Q} \partial_x F - \Av_{Q_L} \partial_x F|^p \cdot \delta_{Q_\ell}^{2-p} \le C_p \sum_{\ell=0}^{L} \int_{Q_\ell} [\Mcal(|\nabla^2 F|)]^p \, dz.
\end{equation}

Again recalling that $\# \Gang(\ell, z^*) \le C$, and that $\delta_Q = \delta_{Q_\ell} = 2^\ell \delta_{Q_0}$ for $Q \in \Gang(\ell, z^*)$, we see that
\begin{equation}\label{eqn:norm5}
    \sum_{\ell=0}^{L-1} \sum_{Q \in \Gang(\ell, z^*)} \delta_Q^{2-p} \le C_p \delta_{Q_L}^{2-p}.
\end{equation}
Note that the condition $p < 2$ is used to sum the geometric series. Now since $Q_L$ is \easy{}, we know that $|z_1 (Q_L) - z_2 (Q_L)| \ge c \delta_{Q_L}$ and $z_1 (Q_L), z_2 (Q_L) \in E^+ \cap CQ_L$. From Corollary \ref{cor:t2}, we have
\begin{equation*}
    \delta_{Q_L}^{2-p} \left| \Av_{Q_L} \p_x F - \frac{f^+ (z_2 (Q_L)) - f^+ (z_1 (Q_L))}{x_2 (Q_L) -x_1 (Q_L)} \right|^p \leq C_p \int_{Q_L} \left[ \Mcal(|\nabla^2 F|) \right]^p dz.
\end{equation*}
Together with \eqref{eqn:norm5}, this yields
\begin{equation*}
    \sum_{\ell=0}^{L-1} \sum_{Q \in \Gang(\ell, z^*)} \left|\Av_{Q_L} \partial_x F - \frac{f^+(z_2 (Q_L)) - f^+(z_1 (Q_L))}{x_2 (Q_L) -x_1 (Q_L)} \right|^p \cdot \delta_{Q}^{2-p} \le C_p \int_{Q_L} [\Mcal(|\nabla^2 F|)]^p \, dz.
\end{equation*}
This equation combined with \eqref{eqn:norm4} and the triangle inequality gives
\begin{equation}\label{eqn:norm6}
    \sum_{\ell=0}^{L-1} \sum_{Q \in \Gang(\ell, z^*)} \left|\Av_{Q} \partial_x F - \frac{f^+(z_2 (Q_L)) - f^+(z_1 (Q_L))}{x_2 (Q_L) -x_1 (Q_L)} \right|^p \cdot \delta_{Q}^{2-p} \le C_p { \sum_{\ell =0}^{L}} \int_{Q_\ell} [\Mcal(|\nabla^2 F|)]^p \, dz.
\end{equation}
For each $Q \in \GP(z^*)$ we have fixed a point $\uz(Q) = (\ux(Q), 0) \in E \cap CQ$. Let $L_Q^{(1)} = f^+ (\uz(Q)) + (\Av_Q \nabla F) \cdot (z - \uz(Q))$. From Remark \ref{rem:t1} and Corollary \ref{cor:t1}, we have for $Q \in \Gang(\ell, z^*)$, 
\begin{equation}\label{eqn:norm7}
    \sum_{|\alpha| \le 1} \delta_Q^{|\al|p - 2p} \int_{1.1Q} |\p^{\al} (F-L^{(1)}_Q)|^p dz \leq C_p \int_{Q_\ell} \left( \Mcal(|\nabla^2 F|) \right)^{p} dz.
\end{equation}
Now let
\begin{equation}\label{eqn:norm8}
    L_Q^{(2)} = f^+ (\uz(Q)) + \frac{f^+(z_2 (Q_L)) - f^+(z_1 (Q_L))}{x_2 (Q_L) -x_1 (Q_L)} \cdot (x - \ux (Q)) + (\Av_Q \p_y F) \cdot y.
\end{equation}
Note that $L_Q^{(1)} - L_Q^{(2)} = \left(\Av_Q \p_x F - \frac{f^+(z_2 (Q_L)) - f^+(z_1 (Q_L))}{x_2 (Q_L) -x_1 (Q_L)} \right) (x - \ux(Q))$, so an elementary integration shows that
\begin{equation}\label{eqn:norm9}
    \sum_{|\al| \le 1} \delta_Q^{|\al|p - 2p} \int_{1.1Q} |\p^\al (L_Q^{(1)} - L_Q^{(2)})|^p \, dz \le C_p \cdot \delta_Q^{2-p} \left|\Av_Q \p_x F - \frac{f^+(z_2 (Q_L)) - f^+(z_1 (Q_L))}{x_2 (Q_L) -x_1 (Q_L)} \right|^p.
\end{equation}
From \eqref{eqn:norm6}, \eqref{eqn:norm7}, \eqref{eqn:norm9}, and $\# \Gang(\ell, z^*) \le C$, we have
\begin{equation}\label{eqn:norm9'}
    \sum_{\ell=0}^{L-1} \sum_{Q \in \Gang(\ell, z^*)} \sum_{|\al| \le 1} \delta_Q^{|\al|p-2p} \int_{1.1Q} |\p^\al (F - L_Q^{(2)})|^p \, dz \le C_p \sum_{\ell=0}^{L} \int_{Q_\ell} [\Mcal (|\nabla^2 F|)]^p \, dz.
\end{equation}
Recall from Lemma \ref{lem:9.4} that $Q_\ell$ for $0 \le \ell \le  L-1$ belongs to $\GP(z^*)$ and every \hard{} $Q \in \GP(z^*)$ belongs to $\Gang(\ell, z^*)$ for some $0 \le \ell \le {L-1}$. Consequently, \eqref{eqn:norm9'} implies the estimate
\begin{equation}\label{eqn:norm10}
    \sum_{\substack{Q \in \GP(z^*)\\Q \ \hard{}}} \sum_{|\al| \le 1} \delta_Q^{|\al|p-2p} \int_{1.1Q} |\p^\al (F - L_Q^{(2)})|^p \, dz \le C_p \sum_{Q \in \GP(z^*)} \int_{Q} [\Mcal (|\nabla^2 F|)]^p \, dz.
\end{equation}

Let $Q$ be a \rel{} $\CZ$ square. If $Q$ is \easy{}, we set $L_Q^\# = \hL_Q$ with $\hL_Q$ as in \eqref{eqn:norm0}. If $Q$ is \hard{}, then set $L_Q^\# = \hL_Q^{(2)}$ as in \eqref{eqn:norm8}, arising from the unique $z^*$ such that $Q \in \GP(z^*)$. In this case, we recall that $z_i (Q) = z_i (\Qbig (z^*)) = z_i (Q_L)$ for $i = 1, 2$, with $Q_L$ as in \eqref{eqn:norm8}. Hence, comparing \eqref{eqn:norm0} with \eqref{eqn:norm8}, we find that
\begin{equation}\label{eqn:norm11}
    L_Q^{\#} = f^+ (\uz(Q)) + \frac{f^+(z_2 (Q)) - f^+(z_1 (Q))}{x_2 (Q) -x_1 (Q)} \cdot (x - \ux (Q)) + (\Av_Q \p_y F) \cdot y.
\end{equation}
for all \rel{} $\CZ$ squares $Q$.

Now every \rel{} $\CZ$ square is either \easy{}, or {is \hard{} and belongs to $\GP(z^*)$ for exactly one $z^*\in E^+$}. Summing \eqref{eqn:norm1} over all \easy{} \rel{} $Q$, and summing \eqref{eqn:norm10} over all $z^* \in E^+$, we obtain that
\begin{align}
    \sum_{Q \in CZ_{\text{rel}}} \sum_{|\al| \le 1} \delta_Q^{|\al|p-2p} \int_{1.1Q} |\p^\al (F - L_Q^{\#})|^p \, dz &= \sum_{\substack{Q \in CZ_{\text{rel}}, \easy{},\\ Q \notin GP(z^*) \; \forall z^* \in E^+}} \sum_{|\al| \le 1} \delta_Q^{|\al|p-2p} \int_{1.1Q} |\p^\al (F - L_Q^{\#})|^p \, dz \nonumber \\ 
    & \quad \quad + \sum_{z^* \in E^+} \sum_{\substack{Q \in CZ_{\text{rel}}, \\Q \in \GP(z^*)}} \sum_{|\al| \le 1} \delta_Q^{|\al|p-2p} \int_{1.1Q} |\p^\al (F - L_Q^{\#})|^p \, dz \nonumber \\
    &\le C_p \sum_{Q \in \CZ} \int_{Q} [\Mcal (|\nabla^2 F|)]^p \, dz \nonumber \\
    &= C_p \int_{Q^0} [\Mcal (|\nabla^2 F|)]^p \, dz \le C_p' \int_{\R^2} |\nabla^2 F|^p \, dz. \label{eqn:norm12}
\end{align}
% \textcolor{red}{To get to the second line, we additionally used the fact that the sets $\Gp(z^*)$ are disjoint across different $z^* \in E^+$ (and these sets are also disjoint from the set of \easy{} CZ cubes), and the third line line follows since the different CZ cubes are disjoint.}{\color{blue} I recommend removing this, but if we keep it, I recommend editing it because ``get to" is a bit informal, and we shouldn't refer to the ``second" or ``third" lines. Instead, I've highlighted a line in blue above that already does what the referee is asking for; I moved it from after \eqref{eqn:norm10} to its current location.}{\color{purple}A: I agree that with the blue highlighted line above, there should be no need to say anything more. We can just point this out in the response to the referee}
Recall that for each $Q \in \CZ$, there are at most $C$ distinct $Q' \in \CZ$ such that $Q \lra Q'$; those $Q'$ satisfy $\frac{1}{2} \delta_Q \le \delta_{Q'} \le 2\delta_Q$. Consequently, \eqref{eqn:norm12} implies that
\begin{equation*}
    \sum_{|\al| \le 1} \sum_{\substack{Q, Q' \in \CZ_{\text{rel}}\\Q \lra Q'}} \delta_Q^{|\al|p-2p} \left\{ \int_{1.1Q} |\p^\al (F - L_Q^{\#})|^p \, dz + \int_{1.1Q'} |\p^\al (F - L_{Q'}^{\#})|^p \, dz \right\} \le C_p \int_{\R^2} |\nabla^2 F|^p \, dz,
\end{equation*}
which in turn implies (by the triangle inequality) that
\begin{equation}\label{eqn:norm13}
    \sum_{|\al| \le 1} \sum_{\substack{Q, Q' \in \CZ_{\text{rel}} \\Q \lra Q'}} \delta_Q^{|\al|p-2p} \int_{1.1Q \cap 1.1Q'} |\p^\al (L_Q^{\#} - L_{Q'}^{\#})|^p \, dz \le C_p \int_{\R^2} |\nabla^2 F|^p \, dz.
\end{equation}
Next, we define
\begin{equation}\label{eqn:norm14}
    L_Q (x, y) := f^+ (\uz(Q)) + \frac{f^+(z_2 (Q)) - f^+(z_1 (Q))}{x_2 (Q) -x_1 (Q)} \cdot (x - \ux (Q))
\end{equation}
for $Q \in \CZ$ \rel{}. In particular, $L_Q (x, y)$ is independent of $y$, and $L_Q (x, y) = L_Q^\# (x, 0)$.
% {\color{blue} I know the referee recommended this, but I think it'd be better to write $L_Q (x, y) = L_Q (x, 0)$ as this is immediate from its definition (10.16), and there's no reason to verify agreement with $L_Q^\#$; it's distracting, I think. Whichever you choose is fine by me though.}

Recall that each $Q \in \CZ$ has the form $I \times [0, |I|]$, $I \times [|I|, 2|I|]$, $I \times [-|I|, 0]$, or $I \times [-2|I|, -|I|]$ for an interval $I$. Hence, Lemma \ref{lem:s2} applies to $Q$, and combined with {\eqref{eqn:norm13}}, \eqref{eqn:norm14}, this yields
\begin{equation}\label{eqn:norm15}
    \sum_{|\al| \le 1} \sum_{\substack{Q, Q' \in \CZ_{\text{rel}}\\Q \lra Q'}} \delta_Q^{|\al|p-2p} \int_{1.1Q \cap 1.1Q'} |\p^\al (L_Q - L_{Q'})|^p \, dz \le C_p \int_{\R^2} |\nabla^2 F|^p \, dz.
\end{equation}
By inspecting \eqref{eqn:norm14}, we see that the left-hand side of \eqref{eqn:norm15} is determined by $f^+$, independently of $F$. 

Conversely, it is easy to exhibit a function $\tF \in L^{2,p} (\Qinner)$ such that $\tF = f^+$ on $E$, and $\norm{\tF}_{L^{2,p} (\Qinner)}^p$ is dominated by the left-hand side of \eqref{eqn:norm15}. To produce $\tF$, we take a Whitney partition of unity $\{ \theta_Q \}_{Q\in \CZ}$ adapted to our Calder\'on-Zygmund decomposition. Specifically, {the $C^2$ functions} $\theta_Q$ satisfy
\begin{equation*}
    \sum_{Q \in \CZ_{\text{rel}}} \theta_Q = 1 \quad \text{ on } \Qinner,
\end{equation*}
and each $\theta_Q$ is supported on $1.1Q$ with $|\p^\al \theta_Q| \le C\delta_Q^{-|\al|}$ for $|\al| \le 2$. Given $f^+ : E^+ \to \R$, we then define
\[
    \tF := \sum_{Q \in \CZ_{\text{rel}}} \theta_Q L_Q \quad \text{ on } \Qinner,
\]
with $L_Q$ given by \eqref{eqn:norm14}.

Note that $\tF$ depends linearly on $f^+$, and that for each $z \in \Qinner$, $\tF(z)$ is determined by ${f^+}|_{S(z)}$ for a subset $S(z) \subset E^+$ with $\# (S(z)) \le C$.

Let us check that $\tF = f^+$ on $E$. Let $Q \in \CZ_{\text{rel}}$. If $1.1Q \cap E \neq \emptyset$, then we have defined $\uz(Q)$ to be the one and only point of $1.1Q \cap E$. Moreover, $L_Q(\uz(Q)) = f^+(\uz(Q))$, by definition \ref{eqn:norm14}. Therefore, $L_Q = f^+$ on $E \cap 1.1Q$ for each \rel{} $Q$. In particular, $L_Q = f^+$ on $E \cap \supp \theta_Q$. Since $\sum_Q \theta_Q = 1$ on $\Qinner$, it follows that
\begin{equation}\label{eqn:norm151}
    \tF(z) = \sum_Q \theta_Q(z) L_Q(z) = \sum_Q \theta_Q(z) f^+(z) = f^+(z) \quad \quad (z \in E \subset \Qinner),
\end{equation}
as claimed.

Now, by the Patching Estimate \eqref{eq:patching}, we see that
\begin{equation*}
    \norm{\tF}_{L^{2,p} (\Qinner)}^p \le C_p \sum_{Q \in \CZ_{\text{rel}}} \norm{L_Q}_{L^{2,p} (1.1Q)}^p + C_p \sum_{|\al| \le 1} \sum_{\substack{Q, Q' \in \CZ_{\text{rel}}\\Q \lra Q'}} \delta_Q^{|\al|p-2p} \int_{1.1Q \cap 1.1Q'} |\p^\al (L_Q - L_{Q'})|^p \, dz.
\end{equation*}
Since each $L_Q$ is a first-degree polynomial, we have $\norm{L_Q}_{L^{2,p} (1.1Q)}^p = 0$ for all $Q \in CZ_{\text{rel}}$. Hence,
\begin{equation}\label{eqn:norm16}
    \norm{\tF}_{L^{2,p} (\Qinner)}^p \le C_p \sum_{|\al| \le 1} \sum_{\substack{Q, Q' \in \CZ_{\text{rel}}\\Q \lra Q'}} \delta_Q^{|\al|p-2p} \int_{1.1Q \cap 1.1Q'} |\p^\al (L_Q - L_{Q'})|^p \, dz.
\end{equation}
Thus, ${T^+} : f^+ \mapsto \tF$ is a linear map from functions $f^+$ on $E^+$ to functions $\tF \in L^{2,p} (\Qinner)$, with $\tF(z)$ determined by the values of $f^+$ from at most $C$ points of $E^+$. Moreover, $T^+$ satisfies $\tF = f^+$ on $E$, as well as \eqref{eqn:norm16}.

Next, we investigate circumstances in which a summand in \eqref{eqn:norm15}, \eqref{eqn:norm16} is identically zero. Those summands arise from $(Q,Q')$ with $Q \lra Q'$ and $Q, Q' \in \CZ_{\text{rel}}$; suppose such $(Q, Q')$ are also both \hard{} and $1.1Q \cap E = 1.1Q' \cap E = \emptyset$. By Lemma \ref{lem:9.5}, we have $Q, Q' \in \GP(z^*)$ for some $z^*$. According to our discussion in Section \ref{sec:special pts}, we have $z_i (Q) = z_i (\Qbig (z^*)) = z_i (Q')$ and $\uz(Q) = \uz(\Qsmall (z^*)) = \uz(Q')$. So $x_i (Q) = x_i (Q')$ and $\ux(Q) = \ux(Q')$, which means by \eqref{eqn:norm14}, we have $L_Q = L_{Q'}$. So the summand in \eqref{eqn:norm15}, \eqref{eqn:norm16} arising from such a pair $(Q, Q')$ is identically zero. {In summary, combining the estimates \eqref{eqn:norm15} and \eqref{eqn:norm16} with the discussion herein, we have proved the following.}

\begin{lem}\label{lem:norm}
    Let
    \begin{equation}\label{eqn:norm}
        \vertiii{f^+}_+^p := \sum_{|\al| \le 1} \sum_{\substack{Q, Q' \in \CZ_{\text{rel}}, \; Q \lra Q'\\Q \text{ or } Q' \; \easy{}, \text{ or }\\1.1Q \cap E \neq \emptyset \text{ or } 1.1Q' \cap E \neq \emptyset }} \delta_Q^{|\al|p-2p} \int_{1.1Q \cap 1.1Q'} |\p^\al (L_Q - L_{Q'})|^p \, dz.
    \end{equation}
    Suppose $\tF \in L^{2,p} (\R^2)$ with $\tF = f^+$ on $E^+$. Then
    \begin{equation*}
        \vertiii{f^+}_+^p \le C_p \norm{\tF}_{L^{2,p} (\R^2)}^p.
    \end{equation*}
    Conversely, the linear map {$T^+:f^+\mapsto \tilde F$ taking functions $f^+$ on $E^+$ to functions $\tilde F\in L^{2,p} (\Qinner)$} defined by \eqref{eqn:norm151} satisfies
    \begin{gather*}
        T^+ f^+ (z) = f^+ (z) \quad (z \in E), \text{ and } \\
        \norm{T^+ f^+}^p_{L^{2,p} (\Qinner)} \le C_p \vertiii{f^+}_+^p.
    \end{gather*}
    Moreover, for any $z \in \Qinner$, $T^+ f^+(z)$ is determined by $f^+\mid_{S(z)}$, for a subset $S(z) \subset E^+$ with $\# (S(z)) \le C$.
\end{lem}

In the definition of $\triplenorm{f}_+$ in the above lemma, there are at most $C \cdot \# E$ summands thanks to Lemma \ref{lem:9.6}. Moreover, from Lemma \ref{lem:zQ in CQ}, we have $\uz(Q) \in CQ$, so we can apply Lemma \ref{lem:t10} to control each summand as follows. 
%M:{I removed ``$Q \leftrightarrow Q'$ for each term in the summation \eqref{eqn:norm}" because the referee asked for verification of this particular hypothesis, and because we repeatedly use that the cubes are adjacent without referencing this fact.}

Since $\partial_y (L_Q - L_{Q'}) = 0$,

\begin{align*}
    &\partial_x (L_Q - L_{Q'}) = \frac{f^+ (z_2 (Q)) - f^+ (z_1 (Q))}{x_2 (Q) - x_1 (Q)} - \frac{f^+ (z_2 (Q')) - f^+ (z_1 (Q'))}{x_2 (Q') - x_1 (Q')}, \text{ and}\\
    &(L_Q - L_{Q'}) (\uz(Q')) = f^+ (\uz(Q')) - f^+ (\uz(Q)) - \frac{\ux(Q') - \ux(Q)}{x_2(Q) - x_1(Q)}  \left(f^+ (z_2 (Q)) - f^+ (z_1 (Q))\right),
\end{align*}
we find that
\begin{equation*}
    \sum_{|\al| \le 1} \delta_Q^{|\al|p - 2p} \int_{1.1Q \cap 1.1Q'} |\p^\alpha (L_Q - L_{Q'})|^p \, dz
\end{equation*}
and
\begin{multline*}
    \delta_Q^2 \left|\delta_Q^{-2} \left( f^+ (\uz(Q')) - f^+ (\uz(Q)) - \frac{\ux(Q') - \ux(Q)}{x_2(Q) - x_1(Q)}  \left( f^+ (z_2 (Q)) - f^+ (z_1 (Q))\right) \right)\right|^p \\
    + \delta_Q^2 \left|\delta_Q^{-1} \left( \frac{f^+ (z_2 (Q)) - f^+ (z_1 (Q))}{x_2 (Q) - x_1 (Q)} - \frac{f^+ (z_2 (Q')) - f^+ (z_1 (Q'))}{x_2 (Q') - x_1 (Q')} \right)\right|^p
\end{multline*}
differ by at most a multiplicative factor $C_p$. {Combining this with the conclusions of Lemma \ref{lem:norm}, we have the following result. We let $X(E^+)$ denote the vector space of real-valued functions on $E^+$.}

\begin{lem}\label{l:data-dep}
    The map {$T^+: X(E^+) \to L^{2,p}(Q_\inner)$} given by $T^+f^+ :=\tilde F$ {as in \eqref{eqn:norm151}} satisfies the following:
    \begin{itemize}
        \item[(i)] $\tilde F = f^+$ on $E$;
        \item[(ii)] $\Vert \tilde F \Vert_{L^{2,p}(Q_{\inner})}^p \leq C_p \sum_{\nu=1}^{\nu_{\max}} \lambda_\nu \vert \ell_\nu^+(f^+)\vert^p$ for some $\lambda_\nu > 0$ and linear functionals $\ell_\nu^+:X(E^+)\to\R$;
        \item[(iii)] Any $F\in L^{2,p}(\R^2)$ with $F=f^+$ on $E^+$ satisfies
        \[
            \Vert F\Vert_{L^{2,p}{(\R^2)}}^p \geq c_p \sum_{\nu=1}^{\nu_{\max}} \lambda_\nu \vert\ell_\nu^+(f^+)\vert^p;
        \]
        \item[(iv)] $\lambda_\nu$ and {$\ell_\nu^+$} depend only on {$E^+$; neither depends on $p$ or $f^+$}.
        \item[(v)] For every $z\in Q_\inner$, there exists a subset $S(z)\subset E^+$ with $\# S(z)\leq C$, such that $\tilde F(z)$ depends only on $f^+|_{S(z)}$; 
        \item[(vi)] For each $\nu\in \{1,\dots,\nu_{\max}\}$, there exists a subset $S_\nu\subset E^+$ with $\#S_\nu\leq C$, such that {$\ell_\nu^+(f^+)$} depends only on {$f^+|_{S_\nu}$};
        \item[(vii)] $\nu_{\max}\leq CN$, where $N=\# E$.
    \end{itemize}
\end{lem}

\section{Removing the extra point from $E^+$} \label{sec:rempt}
{Observe that Lemma \ref{l:data-dep} concludes Theorem \ref{t:main-reduction}, up to getting rid of the extra point $(-1,0)\in E^+\setminus E$.} To do so, we introduce a trivial extension operator as follows.

{Recall that $E\subset [-2^{-11},2^{-11}]\times \{0\}\subset \R^2$ and that $(2^{-11},0), (-2^{-11},0)\in E$.} Given a function $F\in L^{2,p}(Q_\inner)$, we let $L_F$ be the unique first degree polynomial that agrees with $F$ at the points $(-2^{-11},0), \ (2^{-11},0)$ and $(2^{-11}, 2^{-11})$. Then, using Corollaries \ref{cor:t1}-\ref{cor:t3}, the Hardy-Littlewood Maximal Theorem, and that $\delta_{Q_{inner}}=2^{-9}$, we have 

\begin{equation}\label{e:Sobolev-add-endpts}
    \Vert \partial^\alpha(F-L_F)\Vert_{L^p(Q_\inner)} \leq C_p \Vert F \Vert_{L^{2,p}(Q_\inner)} \qquad \forall |\alpha|\leq 1
\end{equation}
and $L_F|_{\R\times \{0\}}$ is determined by $F( 2^{-11},0)$ and $F(-2^{-11},0)$, {with linear dependency}. In particular,
\begin{equation}\label{e:LI-add-endpts}
    L_F(-1,0) = a_+ F(2^{-11},0) + a_- F(-2^{-11},0),
\end{equation}
for universal constants $a_+, a_-$.

Now fix a cutoff $\chi\in C^2(\R^2)$ supported in $Q_\inner$ with $\chi =1$ on $[-2^{-11},2^{-11}]\times\{0\}$ and $\Vert \chi \Vert_{C^2(\R^2)}\leq C$. Define {$\Ecal: L^{2,p}(Q_\inner) \to L^{2,p}(\R^2)$ by}
\[
     \Ecal F = \chi F + (1-\chi)L_F.
\]
By \eqref{e:Sobolev-add-endpts}, \eqref{e:LI-add-endpts} and the definition of $L_F$ and $\chi$, for every $F\in L^{2,p}(Q_\inner)$ we have
\begin{itemize}
    \item[(1)] $\Ecal F= F$ on $[-2^{-11}, 2^{-11}]\times \{0\}$;
    \item[(2)] $\Vert \Ecal F \Vert_{L^{2,p}(\R^2)}\leq C_p \Vert F \Vert_{L^{2,p}(Q_\inner)}$;
    \item[(3)] $\Ecal F(-1,0) = a_+ F(2^{-11},0) + a_- F(-2^{-11},0)$.
\end{itemize}

Now suppose that $f: E\to \R$ is given. Define $f^+:E^+\to \R$ by $f^+=f$ on $E$ and
\[
    f^+(-1,0)=a_+ f(2^{-11},0) + a_- f(-2^{-11},0).
\]
Then define
\[
    \tilde F = T^+ f^+ \in L^{2,p}(Q_\inner), \qquad \ell_\nu(f) = \ell_\nu^+(f^+),
\]
with $T^+$ and $\ell_\nu^+$ as in Lemma \ref{l:data-dep}. Then, the conclusion of Lemma \ref{l:data-dep} implies that $\tilde F = f$ on $E$ and
\[
    \Vert \tilde F \Vert_{L^{2,p}(Q_\inner)}^p \leq C_p \sum_{\nu =1}^{\nu_{\max}} \lambda_\nu \vert \ell_\nu(f) \vert^p.
\]
Moreover, for any $F\in L^{2,p}(Q_\inner)$ with $F=f$ on $E$, letting $F^+:= \Ecal F$ and using properties (1), (2) above together with Lemma \ref{l:data-dep}(iii) yields

\[
    \sum_{\nu=1}^{\nu_{\max}} \lambda_\nu \vert \ell_\nu(f)\vert^p = \sum_{\nu=1}^{\nu_{\max}} \lambda_\nu \vert \ell_\nu^+(f^+)\vert^p \leq C_p \Vert F^+ \Vert_{L^{2,p}(\R^2)}^p \leq C_p \Vert F \Vert_{L^{2,p}(Q_\inner)}^p.
\]
Setting $Tf := T^+ f^+$, we summarize what we have just shown in the following lemma.

\begin{lem}\label{l:main-loc-image}
    Let $E\subset [-2^{-11},2^{-11}]\times \{0\}\subset \R^2$ be finite with $N= \# E$ and $(2^{-11},0), (-2^{-11},0) \in E$. Then there exist a linear map $T:X(E)\to L^{2,p}(Q_\inner)$, a positive integer $\nu_{\max}$, positive coefficients $\lambda_1,\dots,\lambda_{\nu_{\max}}$ and linear functionals $\ell_1,\dots,\ell_{\nu_{\max}}: X(E) \to \R$ such that the following holds. Given any $f\in X(E)$,
    \begin{itemize}
        \item[(i)] $Tf=f$ on $E$;
        \item[(ii)] $\Vert Tf \Vert_{L^{2,p}(Q_\inner)}^p \leq C_p \sum_{\nu=1}^{\nu_{\max}} \lambda_\nu \vert \ell_\nu(f)\vert^p$;
        \item[(iii)] Any $F\in L^{2,p}(Q_\inner)$ with $F=f$ on $E$ satisfies
        \[
            \Vert F \Vert_{L^{2,p}(Q_\inner)}^p \geq c_p \sum_{\nu=1}^{\nu_{\max}} \lambda_\nu \vert \ell_\nu(f)\vert^p;
        \]
        \item[(iv)] $\lambda_\nu$ and $\ell_\nu$ depend only on $E$, not on $p$ or $f$;
        \item[(v)] For every $z\in Q_\inner$, there exists a subset $S(z)\subset E$ with $\# S(z) \leq C$, such that $Tf(z)$ is determined only by $f|_{S(z)}$;
        \item[(vi)] For each $\nu$, there exists a subset $S_\nu\subset E$ with $\# S_\nu \leq C$, such that $\ell_\nu(f)$ is determined only by $f|_{S_\nu}$;
        \item[(vii)] $\nu_{\max} \leq CN$.
    \end{itemize}
\end{lem}

\begin{proof}[Proof of Theorem \ref{t:main-reduction}]
    By translation and dilation, we can assume that the finite set $E$ satisfies the hypotheses of Lemma \ref{l:main-loc-image}. Given $f\in X(E)$, let $T^\sharp f:=(\Ecal\circ T)(f)$, for $T$ given by Lemma \ref{l:main-loc-image}. Then $T^\sharp f\in L^{2,p}(\R^2)$ and $T^\sharp f = Tf = f$ on $E$, with
    \begin{equation}
        \Vert T^\sharp f \Vert_{L^{2,p}(\R^2)}^p \leq {C_p} \Vert Tf \Vert_{L^{2,p}(Q_\inner)}^p \leq C'_p \sum_{\nu=1}^{\nu_{\max}} \lambda_\nu \vert \ell_\nu(f) \vert^p. \label{eq:ef}
    \end{equation}
    In addition, for any $F\in L^{2,p}(\R^2)$ satisfying $F=f$ on $E$ we clearly have
    \[
        \Vert F \Vert_{L^{2,p}(\R^2)}^p \geq \Vert F \Vert_{L^{2,p}(Q_\inner)}^p \geq c_p \sum_{\nu=1}^{\nu_{\max}} \lambda_\nu \vert \ell_\nu(f)\vert^p
    \]
    This verifies conclusions (a),(b) and (c). The validity of property (d) follows from the fact that given any point $z\in \R^2$, $T^\sharp f(z)$ depends only on $Tf$ evaluated at the points $(\pm 2^{-11},0)$, $({2^{-11}}, 2^{-11})$ and $z$ (where the latter dependency is only if $z\in Q_\inner$). Thanks to conclusion (v) of Lemma \ref{l:main-loc-image}, (d) indeed follows.

    In light of (vi) and (vii) of Lemma \ref{l:main-loc-image}, conclusions (e) and (f) follow. This completes the proof.
\end{proof}

\textbf{Conflict of interest statement:} The authors have no conflict of interest to declare.

\textbf{Data availability:} No data was collected as part of the project.

\newpage
\bibliographystyle{amsplain}
\bibliography{main}

\end{document}